\documentclass[11pt]{article}

\usepackage[pdftex]{graphicx}

\usepackage{amssymb}
\usepackage{amsthm}

\usepackage{accents}

\makeatletter
\renewcommand\section{\@startsection{section}{1}{\z@}%
                                  {-3.5ex \@plus -1ex \@minus -.2ex}%
                                  {2.3ex \@plus.2ex}%
                                  {\normalfont\large\bfseries}}
\makeatother

\DeclareGraphicsExtensions{.jpg}
\begin{document}

\title{Generalized domination structure in cubic graphs}

\author{Misa Nakanishi \thanks{E-mail address : nakanishi@2004.jukuin.keio.ac.jp}}
\date{}
\maketitle

\begin{abstract}
In this paper, we consider generalized domination structure in graphs,
which stipulates the structure of a minimum dominating set. Two cycles of length 0 mod 3 intersecting with one path are the constituents of
the domination structure and by taking every three vertices on the cycles we can obtain a minimum dominating set. For a cubic graph, we
construct generalized domination structure by adding edges in a certain way. We prove
that the minimum dominating set of a cubic graph is determined in
polynomial time. 
\end{abstract}

\newtheorem{thm}{Theorem}[section]
\newtheorem{lem}{Lemma}[section]
\newtheorem{prop}{Proposition}[section]
\newtheorem{cor}{Corollary}[section]
\newtheorem{rem}{Remark}[section]
\newtheorem{conj}{Conjecture}[section]
\newtheorem{claim}{Claim}[section]
\newtheorem{fact}{Fact}[section]
\newtheorem{obs}{Observation}[section]

\newtheorem{defn}{Definition}[section]
\newtheorem{propa}{Proposition}
\renewcommand{\thepropa}{\Alph{propa}}
\newtheorem{conja}[propa]{Conjecture}

\section{Notation}
In this paper, a graph $G$ is finite, undirected, and simple with the vertex set $V$ and edge set $E$. We follow \cite{Diestel} for basic notation. 
For a vertex $v \in V(G)$, the open neighborhood, denoted by $N_G(v)$, is $\{ u \in V(G) \colon\ uv \in E(G) \}$, and the closed neighborhood, denoted by $N_G[v]$, is $N_G(v) \cup \{v\}$, also for a set $W \subseteq V(G)$, let $N_G(W) = \bigcup_{v \in W} N_G(v)$ and $N_G[W] = N_G(W) \cup W$. A {\it dominating set} $X \subseteq V(G)$ is such that $N_G[X] = V(G)$. For a set $S \subseteq V (G)$, as is clear from the context, $S$ denotes $G[S]$. A minimum dominating set, called a {\it d-set}, is a dominating set of minimum cardinality.  Two cycles $C_1$ and $C_2$ are said to be {\it connecting without seams} if $C_1 \cap C_2$ is one path. For a graph $G$, {\it structure} $H$ is the union of maximal number of cycles of length 0 mod 3 in $G$ where each cycle of length 0 mod 3 is connecting without seams for some other cycle of length 0 mod 3, or one cycle of length 0 mod 3 in $H$, in addition, if $V(G - H) = \emptyset$, we call this structure {\it domination structure}. Let $\mathcal{F}(G)$ be the set of all structures in a graph $G$. 

\section{Generalized domination structure in cubic graphs}

We consider a connected graph $G$, otherwise consider each component one by one. We introduce the construction scheme ${\bf K}$ as follows. \\

\noindent ${\bf K}$: Input a connected graph $G$. \\
(1) Let $G_0 = G$ and $k = 0$. \\
(2) Let $v$ be a cut vertex of $G_k$.
For every pair of components $C_1$ and $C_2$ of $G_k - v$ and for every pair of vertices $v_1 \in C_1 \cap N_{G_k}(v)$ and $v_2 \in C_2 \cap N_{G_k}(v)$, add an edge $v_1v_2$. Increment $k$. \\
(3) Let $D_2$ be an induced cycle of length 2 mod 3 in $G_k$. Take a vertex $w \in V(D_2)$, and set $N_{D_2}(w) = \{w_1, w_2\}$. Now, add an edge $w_1w_2$. Set $w_1ww_2\alpha \subseteq D_2$. Now, add two edges $w_1\alpha$ and $w\alpha$. Increment $k$. \\
(4) Let $D_1$ be an induced cycle of length 1 mod 3 in $G_k$. Take a vertex $x \in V(D_1)$, and set $N_{D_1}(x) = \{x_1, x_2\}$. Now, add an edge $x_1x_2$. Set $x_1xx_2\alpha \subseteq D_1$. Now, add an edge $x\alpha$. Increment $k$. \\
(5) Repeat (2)-(4). \\
(6) Return the resulting graph $G_{k}$. \\

The next remark is a basic concept of the following proofs.

\begin{rem}\label{R}
Let $X$ be a dominating set of $G$. Every subset $D \subseteq X$ is a d-set of $N_G[D]$ if and only if $X$ is a d-set of $G$. 
\end{rem}

Let ${\bf K}(G)$ be a graph constructed by applying ${\bf K}$ to $G$. Note that ${\bf K}(G)$ is not unique and constructed from $G$ arbitrarily. 

\begin{prop}\label{P1}
${\bf K}(G)$ is domination structure. Moreover, $|\mathcal{F}({\bf K}(G))| = 1$.
\end{prop} 

\begin{proof}
From the rules, ${\bf K}(G)$ is 2-connected. Hence ${\bf K}(G)$ has an ear decomposition. Since all induced cycles are of length 0 mod 3, the domination structure $H$ is obtained by finding 0 mod 3 induced cycles connecting without seams one by one, so that $H = {\bf K}(G)$. We had the claim.
\end{proof}

Suppose that for all $v \in V(G)$, $d_G(v) \geq 3$. 

\begin{fact}\label{p1}
For domination structure $H = {\bf K}(G)$, label every three vertices on the induced cycles that constitute $H$ in order of connecting without seams. Note that a certain induced cycle of $H$ is not counted for the labeling, and may have no labels, where the vertices of the induced cycle are all in other induced cycles of $H$. There exist at most $|V(G)|$ cases of labeling. (i) For every labeling, the set of all labeled vertices is a dominating set of ${\bf K}(G)$. (ii) For at least one labeling, the set of all labeled vertices is a d-set of ${\bf K}(G)$.
\end{fact}

\begin{proof}
By the first vertex choice for the labeling, all labeled vertices are uniquely determined in $V(G)$ since for all $v \in V(G)$, $d_G(v) \geq 3$, and so there exist at most $|V(G)|$ cases of labeling. The statement (i) is obvious. The statement (ii) follows from Remark \ref{R}. 
\end{proof}

Let $Y$ be a d-set of ${\bf K}(G)$ that is obtained by applying Fact \ref{p1}. Let $\mathcal{Y}$ be the set of all $Y$. Let $X$ be a d-set of $G$. Let $\mathcal{X}$ be the set of all $X$. 

\begin{prop}\label{P2}
For some $X \in \mathcal{X}$, and some $Y \in \mathcal{Y}$, $Y \subseteq X$.
\end{prop}

\begin{proof}
For some $Y \in \mathcal{Y}$, if $Y$ is a dominating set of $G$, then for some $X \in \mathcal{X}$, $X = Y$. Otherwise, for all $Y \in \mathcal{Y}$, $Y$ is not a dominating set of $G$. Now, for some $v \in Y$, and some $w \in V(G) \setminus Y$, $vw$ is an added edge for ${\bf K}(G)$ and $N_G(w) \subseteq V(G) \setminus Y$. Now, we consider such $v$ and $w$. Let $Z' = \{ w \in V(G) \setminus Y \colon\ v \in Y, vw \in E({\bf K}(G)) \setminus E(G), N_{G}(w) \subseteq V(G) \setminus Y\}$ and $E' = \{ vw \in E({\bf K}(G)) \setminus E(G) \colon\ w \in V(G) \setminus Y, v \in Y, N_{G}(w) \subseteq V(G) \setminus Y\}$. 
Let $E''(vw)$ be an induced cycle of length 0 mod
3 in ${\bf K}(G)$ such that for $vw \in E'$, $vw \in E''(vw)$ holds.
Let $\mathcal{E}(vw)$ be the set of all $E''(vw)$ for $vw \in E'$ and let $\mathcal{E} = \bigcup_{vw \in E'}\mathcal{E}(vw)$. Let $J$ be the union of all induced cycles of length 0
mod 3 in ${\bf K}(G)$ other than the induced cycles in $\mathcal{E}$. By the definition of $Y$ and Remark \ref{R}, $J \cap Y$ is a d-set of $N_G[J \cap Y]$, and $Y$ is a d-set of $G - Z' = N_G[Y]$.  
Let $W$ be a subset of $V(G) \setminus Y$ of minimum cardinality such that $Y \cup W$ is a dominating set of $G$. By the definition of $W$, $W$ is a d-set of $N_G[W]$. Since for all $x \in V(G)$, $d_G(x) \geq 3$, and by Fact \ref{p1}, $E(G[Y]) = \emptyset$, and $Y \cup W$ is a minimal dominating set of $G$. Therefore, by Remark \ref{R}, for some $X \in \mathcal{X}$, $X = Y \cup W$.
\end{proof}

Suppose that $Y$ is a d-set of ${\bf K}(G)$ that satisfies Proposition \ref{P2}.

\begin{fact}\label{P4}
Let $G'$ be constructed by deleting $Y$, and for every pair $w_1, w_2 \in \bigcup_{y \in Y}N_G(y)$, adding an edge $w_1w_2$ to $G$. Let $Z_1$ be a d-set of $G'$. Let $G'' = G - N_G[Y]$ and $Z_2$ be a d-set of $G''$. If $|Z_1| < |Z_2|$, then $Y \cup Z_1$ is a d-set of $G$, and $Z_1 \cap \bigcup_{y \in Y}N_G(y) \ne \emptyset$. If $|Z_1| \geq |Z_2|$, then $Y \cup Z_2$ is a d-set of $G$. 
\end{fact}

\begin{proof}
Obviously, $Y \cup Z_2$ is a d-set of $G$ if and only if $|Z_1| \geq |Z_2|$.
For a set $A \subseteq V(G)$, suppose that $Y \cup A$ is a d-set of $G$. Let $A' = A \setminus \bigcup_{y \in Y}N_G(y)$. By Remark \ref{R}, $A'$ is a d-set of $N_G[A']$. Suppose that $|Z_1| < |Z_2|$, then $A'$ is not a dominating set of $G''$, and $A \cap \bigcup_{y \in Y}N_G(y) \ne \emptyset$. Now, $A$ is a minimal dominating set of $G'$. By the definition of $G'$, $Y \cup Z_1$ is a minimal dominating set of $G$, and so it suffices that $A = Z_1$. Thus $Z_1 \cap \bigcup_{y \in Y}N_G(y) \ne \emptyset$. 
\end{proof}

Suppose that $G$ is cubic.

\begin{thm}
For some $X \in \mathcal{X}$, $X$ is determined in polynomial time.
\end{thm}

\begin{proof}
By Proposition \ref{P2}, $Y \subseteq X$ for some $X \in \mathcal{X}$. Let $G_0 = G$. Let $G_1$ be constructed by deleting $Y$, and for every pair $w_1, w_2 \in \bigcup_{y \in Y}N_G(y)$, adding an edge $w_1w_2$ to $G_0$. Let $G_2 = G_0 - N_{G_0}[Y]$. Let $W_1$ be a d-set of $G_1$ and $W_2$ be a d-set of $G_2$. Since $G_0$ is cubic and by the definition of $G_2$, each component of $G_2$ is path or cycle. Thus $W_2$ is determined in polynomial time. Suppose that $|W_1| < |W_2|$. Let $Y_1$ be a d-set of ${\bf K}(G_1)$ that satisfies Proposition \ref{P2}. By the definition of $Y_1$, it suffices that $Y_1 \cap \bigcup_{y \in Y}N_{G_0}(y) \ne \emptyset$. Let $G_3$ be constructed by deleting $Y_1$, and for every pair $w_1, w_2 \in \bigcup_{y \in Y_1}N_{G_1}(y)$, adding an edge $w_1w_2$ to $G_1$. Let $G_4 = G_1 - N_{G_1}[Y_1]$. Let $W_3$ be a d-set of $G_3$ and $W_4$ be a d-set of $G_4$. By the definition of $G_4$, each component of $G_4$ is path. Thus $W_4$ is determined in polynomial time. Suppose that $|W_3| < |W_4|$. Let $Y_2$ be a d-set of ${\bf K}(G_3)$ that satisfies Proposition \ref{P2}. By the definition of $Y_2$, it suffices that $Y_2 \cap \bigcup_{y \in Y_1}N_{G_1}(y) \ne \emptyset$. Let $G_5$ be constructed by deleting $Y_2$, and for every pair $w_1, w_2 \in \bigcup_{y \in Y_2}N_{G_3}(y)$, adding an edge $w_1w_2$ to $G_3$. Let $G_6 = G_3 - N_{G_3}[Y_2]$. Let $W_5$ be a d-set of $G_5$ and $W_6$ be a d-set of $G_6$. By the definition of $G_6$, $G_6$ is independent. Thus  $W_6 = V(G_6)$. Suppose that $|W_5| < |W_6|$. Let $Y_3$ be a d-set of ${\bf K}(G_5)$ that satisfies Proposition \ref{P2}. By the definition of $Y_3$, it suffices that $Y_3 \cap \bigcup_{y \in Y_2}N_{G_3}(y) \ne \emptyset$. Now, $Y_3$ is a dominating set of $G_5$ and so it suffices that $W_5 = Y_3$. By Fact \ref{P4}, if $|W_5| < |W_6|$, then it suffices that $W_3 = Y_2 \cup W_5$, otherwise, it suffices that $W_3 = Y_2 \cup W_6$. By Fact \ref{P4}, if $|W_3| < |W_4|$, then it suffices that $W_1 = Y_1 \cup W_3$, otherwise, it suffices that $W_1 = Y_1 \cup W_4$. By Fact \ref{P4}, if $|W_1| < |W_2|$, then it suffices that $X = Y \cup W_1$, otherwise, it suffices that $X = Y \cup W_2$. The proof is complete.
\end{proof}

\end{document}